\documentclass{article}

\usepackage{titlesec}

\titleformat{\section}[hang]{\normalfont\Large\bfseries}{\thesection}{1em}{}
\titleformat{\subsection}[hang]{\normalfont\large}{\thesubsection}{1em}{}

\usepackage[paper=a4paper,left=25mm,right=25mm,top=25mm,bottom=25mm]{geometry}
\usepackage{a4,amsfonts,latexsym,amscd,amsopn,amssymb,amsmath, amsthm,bbm,stmaryrd,wasysym,mathbbol}

\usepackage[english]{babel}
\usepackage[utf8]{inputenc}
\UseRawInputEncoding
\usepackage{setspace}
\usepackage{amscd}
\usepackage{bbm}
\usepackage{mathtools}
\usepackage{enumitem}
\usepackage{xcolor}
\allowdisplaybreaks[1]

\newtheorem{Theorem}{Theorem}[section]
\newtheorem{Lemma}[Theorem]{Lemma}
\newtheorem{Corollary}[Theorem]{Corollary}

\newtheorem{Remark}[Theorem]{Remark}
\newtheorem{Proposition}[Theorem]{Proposition}

\theoremstyle{definition}
\newtheorem{Definition}[Theorem]{Definition}

\newtheorem{Assumption}[Theorem]{Assumption}

\usepackage{amssymb}

\newcommand\CR{{\mathcal  R}}
\newcommand\R{{\mathbb R}}
\newcommand\X{{\R^d}}

\newcommand\N{{\mathbb N}}
\newcommand\CM{{\mathcal M}}

\newcommand\B{{\mathcal B}}

\newcommand\La{\Lambda}
\newcommand\la{\lambda}
\newcommand\al{\alpha}

\newcommand\Ga{\Gamma}
\newcommand\ga{\gamma}

\newcommand{\K}{{\mathbb{K}}}

\usepackage[colorinlistoftodos]{todonotes}

\newcommand{\Z}{\mathbb Z}

\providecommand{\keywords}[1]
{
	\small	
	\textit{Keywords:} #1
}

\title{Gibbs measure over the cone of vector-valued discrete measures}

\author{First Author Name$^{a}$$^{*}$, Second Author Name$^{b}$$^{c}$, etc.$^{a}$$^{c}$ \\
	\small $^{a}$Department, University, City, Country \\
	\small $^{b}$Department, University, City, Country \\
	\small $^{c}$Department, University, City, Country \\\\
	\small $^{*}$Corresponding author: first name, initials, 
}
\author{Luca Di Persio \footnote{College of Mathematics, 
		Department of Computer Science, University of Verona, Strada le Grazie 15 - 37134 Verona - Italy, luca.dipersio@univr.it}
	\and  Yuri Kondratiev \footnote{This work has been developed while Yuri was a member of the Dragomanov University, Kyiv.}
	\and Viktorya Vardanyan \footnote{Department of Mathematics, University of Trento, Via Sommarive 14-38123 Povo(TN)-Italy, viktorya.vardanyan@unitn.it}
}

\date{} 

\begin{document}
	\thispagestyle{empty}
	\maketitle
	
	\begin{abstract}  
	
	We consider a gas whose each particle is characterised by a pair $(x,v_x)$ with the 
position $x\in \mathbb R^d$ and the velocity $v_x\in \mathbb R^d_0= \mathbb R^d\setminus \{0\}$. 
We define Gibbs measures on the cone of vector-valued measures and aim to prove their existence. We introduce the family of probability measures $\mu_\lambda$  on the cone $\mathbb K(\mathbb R^d)$. We define local Hamiltonian and partition functions for a positive, symmetric, bounded and measurable pair potential. Using those above, we define Gibbs's measure as a solution to the Dobrushin-Lanford-Ruelle equation. In particular, we focus on the subset of tempered Gibbs measures. To prove the existence of the Gibbs measure, we show that the subset of tempered Gibbs measures is non-empty and relatively compact.

	\end{abstract} \hspace{10pt}


\section{Introduction}

A classical free gas is the simplest model in which the interactions between the atoms are entirely ignored, and the atoms are treated as classical particles. Poisson measures give the equilibrium states of such gases.  The states of interacting gases are known as Gibbs measures, and Dobrushin, Lanford, and Ruelle introduced the existence problem of such measures. The Dobrushin-Lanford-Ruelle(DLR) theory involves the use of perturbation techniques to study the behaviour of a system in the presence of interactions, see \cite{DO1}, \cite{DO2}, \cite{LR}, \cite{RU1}, \cite{RU2}.  Kondratiev, Pasurek, R{\"o}ckner \cite{KP} presented the existence of Gibbs measures associated with classical particle systems in the continuum. The existence of Gibbs measures over the cone of discrete Radon measures on $\mathbb R^d$ is studied by  Kondratiev, Pasurek, R{\"o}ckner, Hagedorn \cite{HK}. They consider interacting particle systems in the continuum $\mathbb R^d, d \in \N $, in which each particle $x$ is attached a characteristic mark $s_x$. We aim to study the generalisation for the case of the cone of vector-valued Radon measures. We consider a gas whose each particle is characterised by a pair $(x,v_x)$ with the position $x\in \X$ and the velocity $v_x\in \R_0^d= \X\setminus \{0\}$. 
 We formulate the problem of Gibbs measures on the cone of vector-valued measures
$$\K(\X)= \bigg\{\eta=\sum_{x\in \tau(\eta)} v_x\delta_x \in \CM(\X) \bigg| \tau(\eta)= \{x\;|\; \eta(\{x\})=:v_x\}\subset \X\bigg\}.$$\par
We aim to prove the existence of Gibbs measure on the aforementioned cone. We consider the case of a positive pair potential $\phi(x, y)$ having finite range. 
 We are going to restrict ourselves to a set $G^t$ of tempered Gibbs measures. Gibbs measure $\ga \in G(\phi)$ is constructed as a cluster point of the net of specification kernels $\{\pi_\Lambda\}_ {\Lambda \in \mathcal{B}_c(\mathbb{R}^d)\}}$.  The essential part of the proof of the main result is to prove the equicontinuity of the local specification kernels. For this purpose, we demonstrate the exponential integrability of a specific Lyapunov functional,  restricted to a small cube $Q_k \subset \mathbb R^d$ and then extend the obtained estimate to large volumes using the consistency property of the local Gibbs specification. We have to emphasise the weak dependence of such bounds on the boundary values $\xi \in \mathbb K(\X)$.
 Then, using the local equicontinuity and (Dobrushin-Lanford-Ruelle) equilibrium equation, the existence of the Gibbs measure will be proved.
 The structure of the paper is following. In Section 2, we present configuration spaces. We introduce the family of probability measures $\mu_\lambda$ on $\K(\X)$ and their properties. We define pair potential and partition of the space $\X$. In Section 3, we define Hamiltonian and partition functions and find bounds for them. In Subsection 3.2, we introduce local Gibbs specification and the corresponding Gibbs measures as solutions to the DLR equation. Moreover, we introduce the set of tempered configurations and the set of tempered Gibbs measures. In Section 4, we formulate our main Theorem \ref{Main Result} on the existence. To show that \( G^t(\phi) \) is non-empty, we construct a measure \( \gamma \) by taking the weak limit of a sequence of finite-volume Gibbs measures with suitable boundary conditions.  In Subsection 4.1, we prove
 the exponential integrability of the following map $ \mathbb k(\X) \ni \eta \mapsto \beta \eta (\mathcal{Q}_k)^2$, for each $k \in  \mathbb Z$, which has a role of Lyapunov functional. Then, we prove the weak dependence of the specification kernels on boundary conditions, both for small volumes $\Lambda:= \mathcal{Q}_k$ and in the thermodynamic limit $ \Lambda \nearrow \X$. In Subsection 4.2, we use the results obtained in Subsection 4.1 to prove equicontinuity of the local specification kernels $\{\pi_\Lambda(d\eta| \xi) \ | \  \Lambda \in \mathcal{Q}_c(\mathbb{R}^d)\}$, with fixed tempered boundary condition  $\xi \in \K^t(\X)$. Then, using the local equicontinuity and (Dobrushin-Lanford-Ruelle) equilibrium equation, we prove the existence of the Gibbs measure.

\section{Description of the model}
\subsection{Preliminaries on configuration spaces}
\vspace{10pt}
Consider the Euclidean space $\X$ and the locally compact pointed space $\R^d_0=\X\setminus \{0\}$. Let $\Ga(\R^d_0 \times \X)$ be the configuration space over 
$\R^d_0\times \X$, i.e., the set of locally finite (i.e.,  finite in any compact)  subsets $\zeta=\{(v,x)\}\subset \R^d_0 \times \X$.  We equip this space with the vague topology \cite{AKR}. Corresponding $\sigma$-algebra is denoted by $\B(\Ga(\R^d_0 \times \X))$. 
Between all configurations, we consider subspace $\Ga_{p}(\R^d_0 \times \X)$
	of so-called pinpointed configurations. 
	By definition, $\Ga_{p}(\R^d_0 \times \X)$ consists of all configurations
	$\zeta \in \Ga (\R^d_0 \times \X)$ such that if $(v_1, x_1), (v_2,x_2)\in \Ga(\R^d_0 \times \X)$
	and $ (v_1,x_1)\neq (v_2,x_2)$ then $x_1\neq x_2$. Thus, the configuration
	can not contain two points  $(v_1,x), (v_2,x)$ with $v_1=v_2$. As easily seen
	$\Ga_{p}(\R^d_0 \times \X)\in \B(\Ga (\R^d_0 \times \X))$. For $\zeta\in \Ga_{p}(\R^d_0 \times \X)$
	we can write a representation
	$$
	\zeta =\{ (v_x, x)\} ,\; \tau(\zeta)= \{x\;|\; (v_x,x)\in \zeta\}\subset \X.
	$$
	For any $\zeta\in \Ga_p (\R^d_0 \times \X)$ and a compact $\La\in \B_c(\X)$ introduce
	a local velocity functional
	$$
	V_\La (\zeta) = \sum_{x\in\tau(\zeta)\cap \La} |v_x|\leq \infty.
	$$
	Define the space $\Pi(\R^d_0\times\X)\subset  \Ga_p (\R^d_0 \times \X)$ 
	
	$$
	\Pi(\R^d_0\times\X) =\{ \zeta\in \Ga_p (\R^d_0 \times \X) \; \forall  \La\in \B_c(\X) \;|\;  V_{\La}(\zeta)<\infty\}.
	$$
	Denote by $\CM(\X,\X)$ the set of all vector-valued Radon measures on $\B(\X)$, i.e.,
	$$
	\CM(\X,\X) \ni \mu: \B(\X) \to \X,\; \forall \La\in \B_c(\X)\;\;|\mu(\La)| <\infty.
	$$
	Introduce the mapping
	$$
	\Pi(\R^d_0\times\X) \ni \zeta \to \eta= \CR\zeta =\sum_{x\in \tau(\zeta)} v_x\delta_x \in \CM(\X)
	$$
	and denote by $\K(\X)$ the image set
	$$
	\K(\X)= \CR(\Pi(\R^d_0\times\X)) \subset \CM(\X).
	$$
 $$\K(\X)= \bigg\{\eta=\sum_{x\in \tau(\eta)}v_x\delta_x \in \CM(\X) \bigg| \tau(\eta)= \{x\;|\; \eta(\{x\})=:v_x\}\subset \X\bigg\}.$$
The topology on $\K(\X)$ is the induced one from $\Ga(\R^d_0, \X)$. By Kondratiev, Di Persio and Vardanyan, new analytical tools are developed to study the mentioned models \cite{KD}. 
\vspace{10pt}

\subsection{ Measures on $\K(\X)$}
\vspace{10pt}
We construct  measures on $\mathbb K(\X)$ as the images under the reflection map $\CR$ of proper measures on the space $\Pi(\R^d_0\times\X)$. To this end, we start with the simplest examples related to Poisson measures. \par
Let us fix a non-atomic Radon measure $\la$ on $\R^d_0$ with following properties:
	$$
	\la(\R^d_0) =\infty,
	$$
	$$
	\forall n\in \N\;\;  \int_{\R^d_0} |v|^n \la(dv) <\infty.
	$$
	As a concrete example of such measure, we will have in mind
	$$
	\la(dv)= \frac{1}{|v|^d} e^{-|v|^2} dv.
	$$
	This measure modifies the usual Maxwell distribution for velocities in statistical physics
	by a singularity at $0$. A wider class of examples of such measures contains elements of the form
	$$
	\la(dv) = \frac{1}{|v|^\al} e^{-|v|^\beta} dv,
	$$
	with $\al\in [d, d+1), \beta >0$.\par
	
	Denote by $m(dx)$ the non-atomic Lebesgue measure on $\X$ and introduce the measure
	$$
	\sigma(dv,dx)= \lambda(dv) m(dx)
	$$
    on $\R^d_0 \times \X$, which is called an intensity measure.
    For any $\La \in \B_c(\R^d_0\times \X)$ we have the following disjoint decomposition 
 $$\Ga(\La)=\bigcup_{n=0}^{\infty}\Ga^n(\La),$$
 where $\Ga(\La):=\{\zeta \in \Ga(\R^d_0 \times \X): \zeta \subset \La\}$ is the set of all configurations supported in $\La$ and  $\Ga^n(\La):=\{\zeta \in \Ga(\La): |\zeta|=n \}$ is the set of $n$-point configurations.
 Lebesgue-Poisson measure on $\Ga(\La)$ is defined as 
 $$\mathcal{L}_\sigma:= \sum_{n=0}^\infty \frac{1}{n!}\sigma^{(n)},$$
 where $\sigma^{(n)}:=(\la \otimes m)^{\otimes n} \circ sym_n^{-1}$ is the measure on $\Ga^n(\La)$. Poisson measure on $\Ga(\La)$ is defined as $\pi_\sigma^\La:= e^{-\sigma(\La)}\mathcal{L}_\sigma$.
    
    There is a standard definition of a Poisson measure 
	$\pi_\sigma$ 
	on $\Ga(\R^d_0 \times \X)$. Then we can proceed as follows. For any $\psi\in C_0(\R^d_0 \times \X) 
	$ (continuous functions with compact supports) define
	$$
	<\psi, \zeta>=\sum_{(v,x)\in \zeta} \psi(v,x),\;\; \zeta \in \Ga(\R^d_0 \times \X).
	$$
	The Poisson measure $\pi_{\sigma} $ is defined via its Laplace transform
	$$
	\int_{\Ga(\R^d_0 \times \X)} e^{<\psi,\zeta>} \pi_{\sigma}  (d\zeta)=
	\exp{\int_{\R^d_0 \times \X} (e^{\psi(v,x)} -1) } \la(dv) m(dx).
	$$
	As a consequence of this definition we have
	$$
	\int_{\Ga(\R^d_0 \times \X)} <\psi,\zeta> \pi_\sigma (d\zeta) = \int_{\R^d_0 \times \X} \psi(v,x) \la(dv) m(dx),
	$$
	and this relation can be extended to the case of any $\psi\in L^1 (\la\otimes m)$.
	
	The Poisson measure $\pi_{\sigma}$ is concentrated on $\Pi(\R^d_0\times\X)$, i.e.,
		
		$$
		\pi_{\sigma} (\Pi(\R^d_0\times\X))=1.
		$$
To obtain measures on $\K(\X)$, we use the pushforward of measures on $\Pi(R^d_0\times \X)$ under the mapping $\CR$.\par 
Denote $\mu_\la$ the image measure on $\K(\X)$ under the reflection map corresponding to $\pi_\sigma$. For $h\in\X$ and $\psi\in C_0(\X) $, introduce a function
	$L_{h,\psi}: \K(\X) \to \R$ via
	$$
	L_{h,\psi}(\eta)=  <h\otimes \psi, \eta> = \int_{\X} \psi(x) <h,\eta(dx)> = \sum_{x\in\tau(\eta)} \psi(x) <h,v_x>.
	$$
Then the definition of the Poisson measure implies
	\begin{equation}
 \label{measure1}
	\int_{\K} e^{<h\otimes \psi, \eta>} \mu_\la (d\eta)= \exp( \int_{\R^d_0\times \X} (e^{<h,v>\psi(x)} -1) \la(dv) m(dx)).
	\end{equation}
Denote 
	$$
	\Psi_\la^h (r) :=  \exp({\int_{\R^d_0} (e^{<h,v>r} -1) \la(dv)}), \;\;r\in\R,
	$$
then 
\begin{equation}
\label{2}
\int_{\K} e^{<h\otimes \psi, \eta>} \mu_\la (d\eta)= e^{\int_{\X} \log (\Psi_\la^h (\psi(x))) m(dx)}.
\end{equation}
The latter relation can be considered as the definition of the measure $\mu_\la$ via its Laplace transform. \par
Let $\Lambda \subset B_c(\X \times \R_0^d) $. We will consider the cone $\mathbb{K}(\Lambda) \in B_c(\mathbb{K}(\X)) $, which consists of discrete vector-valued measures with support $\Lambda$. 
Let $\mu_\la$ be a probability measure on $(\K(\X),\B(\K(\X)))$.
For $\eta\in\K(\X)$ of the form $\eta=\sum_{x\in\tau(\eta)}v_x\delta_x$, there is a  canonical projection 
\begin{equation}
\label{canonical}
    \mathbb P_\Lambda:\mathbb K(\X)\ni \eta \mapsto \eta_ \Lambda:=\sum_{x \in \tau(\eta)\cap \Lambda}v_x \delta_x \in \mathbb K(\Lambda).
\end{equation}
The measure on $\mathbb{K}(\Lambda)$ is defined as
$$\mu_\la^\Lambda:=\mu_\la\circ \mathbb P_\Lambda^{-1},$$
and is characterized via its Laplace transform similar to \eqref{measure1}
\begin{equation}
\label{laplace1}
 \int_{\K(\Lambda)} e^{<h\otimes \psi, \eta>} \mu_\la^\Lambda (d\eta)= \exp\bigg( \int_{\Lambda} (e^{<h,v>\psi(x)} -1) \la(dv) m(dx)\bigg).
 \end{equation}
Denote 
	$$
	\Psi_\la^h (r) := \exp({\int_{\Lambda} (e^{<h,v>r} -1) \la(dv)}), \;\;r\in\R,
	$$
 equivalently we will have
 
\begin{equation}
\label{local 2}
	\int_{\K(\Lambda)} e^{<h\otimes \psi, \eta>} \mu_\la^\Lambda (d\eta)= e^{\int_{\Lambda} \log (1+\Psi_\la^h (\psi(x))) m(dx)}.
\end{equation}

The random measure $\mu_\la$ has independent increments, i.e $\eta(\Lambda_1), ..., \eta(\Lambda_N)$ are independent for any $N \in \N$ and disjoint $\Lambda_1,..., \Lambda_N \in B_c (\X).$ Or 
    \begin{equation}
    \label{independence}
        \int_{\K(\R^d)}\prod_{i=1}^N L_i^{h,\psi}(\eta(\Lambda_i))\mu_\la(d\eta)=\prod_{i=1}^N\int_{\K(\R^d)}L_i^{h,\psi}(\eta(\Lambda))\mu_\la(d\eta) 
    \end{equation}
     for any collection of $L_i^{h,\psi}\in L^\infty$, $ 1 \leq i \leq N.$\par
All local polynomial moments exist, i.e. for $n \in \N$
    \begin{equation}
    \label{moments}
    \mathbb{E}_{\mu_\la}[| \langle h \otimes \psi, \cdot \rangle|^n] \leq n!||\Psi_\la^h (\psi(x)))||_\infty ^n max\{1;  m(\Lambda)\}^n < \infty, 
    \end{equation}
where $\psi :\X \to \R$ is bounded and supported by $ \Lambda \in B_c(\X)$.
\vspace{10pt}

\subsection{Pair potential and partition of the space $\X$}
\vspace{10pt}
In this subsection, we introduce pair potential and the conditions it satisfies. Moreover, we introduce the partition of $\X$.
\begin{Assumption}
\label{pair potential}
Let
\begin{equation} 
    \phi:\X\times \X \to \mathbb{R}_0^{+}
\end{equation}
be positive, symmetric, bounded and measurable pair potential. Denote
$$0 \leq \sup _{x, y \in \mathbb{R}^d}\{|\phi(x, y)|\}=:\|\phi\|_{\infty}<\infty.$$
We suppose $\phi$ satisfies \textbf{Finite Range (FR)} condition, i.e. $\exists R \in (0, \infty)$ such that 
$$\phi(x,y)=0 \quad if \quad |x-y|>R,$$
and \textbf{Repulsion Condition (RC)}, i.e. $\exists \delta>0$ such that 
$$A_\delta:=\inf_{\substack{x,y \in \X \\ |x-y| \leq \delta}}\phi(x,y)> 0.$$

\end{Assumption}
The (FR) ensures that the potential energy is non-zero only when the particles are within a certain range of each other.
 The (RC) condition ensures a repulsive force between particles when they are very close. It prevents the particles from collapsing into each other and ensures that they maintain a minimum distance from each other. The bound $||\phi||_{\infty}$ indicates that the pair potential is uniformly bounded over all pairs of particles, implying that the interaction energy cannot become arbitrarily large.
 Summing up, the given conditions guarantee that the pair potential $\phi$ describes a physically reasonable interaction between particles, with a finite range and a repulsive component that prevents particles from collapsing into each other.

We are going to take a partition of $\X$ as in \cite{HK}.
 We consider the cubes indexed by $k \in \mathbb{Z}^d$
$$\mathcal{Q}_k:= [-\frac{1}{2}g, \frac{1}{2}g)^d+gk \subset \X, $$
with parameter $g:=\delta / \sqrt{d}$, where $\delta>0$ is such that the repulsion condition (RC) holds. These cubs form the partition of $\X$.
Cubes are centered at the point $gk$ with edge length $g>0$, Lebesgue volume $m(\mathcal{Q}_k)=g^d$ and diameter 
$$diam(\mathcal{Q}_k):=\sup_{x,y \in \mathcal{Q}_k}|x-y|_\X=\delta, $$
which means $\phi(x,y) \geq A_\delta$ for all $x,y \in \mathcal{Q}_k $.
We define also the family of neighbour cubes of $\mathcal{Q}_k$ for each 
$k \in \mathbb{Z}^d$ denoted by

\begin{equation}
\label{neighbour}
    \partial_\delta^\phi k:=\{ j \in \mathbb{Z}^d \backslash \{k\} \ | \ \exists x \in \mathcal{Q}_k, \exists y \in \mathcal{Q}_j: \ \phi(x,y) \neq 0  \}.
\end{equation}
The number of such neighbour cubes for every $\mathcal{Q}_k , \ k \in \mathbb{Z}^d$ can be estimated by 

\begin{equation}
\label{estimation}
  \sup_{k \in \mathbb{Z}^d}\partial_\delta^\phi k \leq m_\delta^\phi,   
\end{equation}
where $$m_\delta^\phi:=v_d d^{d/2}[R/\delta+1]^d  \quad with \quad v_d:=\dfrac{\pi^{d/2}}{\Gamma(d/2+1)}$$
is an interaction parameter.\\
To each index set $\mathcal{K}\Subset \mathbb{Z}^d$ there corresponds 
$$\Lambda_\mathcal{K}:=\bigsqcup_{k \in \mathcal{K}} \mathcal{Q}_k \in B(\X),$$
by $\mathcal{Q}_c(\X)$ is denoted the family of all such domains.\\
For $\Lambda \in B(\X) $ we define 
$$\mathcal{K}_\Lambda:=\{ j \in \mathbb{Z}^d \ | \  \mathcal{Q}_j \cap \Lambda \neq \varnothing \},$$
with   $|\mathcal{K}_\Lambda|$ number of cubes $\mathcal{Q}_k$ having non-void intersection with $\Lambda$ and
$$|\mathcal{K}_\Lambda|< \infty, \ \forall \Lambda \in B_c(\X).$$

\section{Local Gibbs specification}
\subsection{Hamiltonian and partition function}
\vspace{10pt}
We define local Hamiltonian and partition function, using the assumption on pair potential and partition introduced above, we obtain bounds for them. 
\begin{Definition}
For each $\eta=\sum_{x\in \tau(\eta)} v_x\delta_x$, $\xi=\sum_{y\in \tau(\xi)} v_y\delta_y$ and $\Lambda \in B_c(\X )$, relative energy (Hamiltonian) is given as  
\begin{equation}
\label{Hamiltonian}
    H_{\Lambda}(\eta| \xi):= \int_{\Lambda} \int_{\Lambda} \phi(x,y)\eta(dx)\eta(dy)+2\int_{\Lambda^c} \int_{\Lambda} \phi(x,y)\eta(dx)\xi(dy)
\end{equation}
or equivalently
$$ H_{\Lambda}(\eta| \xi):= \sum_{x,x' \in \tau(\eta)\cap \Lambda}\phi(x,x')(v_x,v_{x'})+2\sum_{\substack{x\in \tau(\eta)\cap \Lambda \\ y \in \tau(\xi )\cap \Lambda^c }}\phi(x,y)(v_x,v_y).$$
\end{Definition}

\begin{Lemma}
    For all $\eta, \xi \in \K(\X)$ and $\Lambda \in B_c(\X)$ the relative energy is finite, i.e.,
    $|H_{\Lambda}(\eta| \xi)|< \infty$.
\end{Lemma}
\begin{proof}
    Considering
    $$ |H_{\Lambda}(\eta| \xi)| \leq \eta(\Lambda)\eta(\Lambda)||\phi||_\infty+2\eta(\Lambda)\xi(\mathcal{U}_\Lambda)||\phi||_\infty,$$
    where $$\mathcal{U}_\Lambda:= \bigsqcup_{k \in \Z^d} \{\mathcal{Q}_k \ | \ \partial_\delta^\phi k \cap \mathcal{K}_\Lambda \neq \varnothing \} \cap \Lambda ^c \in B_c(\X) $$
    and having $\eta, \xi \in \K(\X)$, the claim follows.
\end{proof}

By the following lemma we aim show the lower bound for local Hamiltonian.
\begin{Lemma}
\label{lower bound}
  Let Assumption \ref{pair potential} hold, then for each $\eta, \xi \in \mathbb{K}\left(\mathbb{R}^d\right)$ and $\Lambda \in B_c\left(\mathbb{R}^d\right)$
  \begin{equation}
      \label{lower bound hamiltonian}
    H_{\Lambda}(\eta| \xi) \geq A \sum_{j \in \mathcal{K}_\Lambda}\eta_\Lambda(\mathcal{Q}_j)^2.
  \end{equation}  
  For each $k \in \Z^d$, we have
  \begin{equation}
      \label{lower bound k}
    H_{\mathcal{Q}_k}(\eta| \xi) \geq A \sum_{x\in\mathcal{Q}_k} \eta(x)^2 ,
  \end{equation} 
when $\xi=0$,
\begin{equation}
      \label{lower bound xi0}
    H_{\mathcal{Q}_k}(\eta_k| 0)  \geq A \sum_{x\in\mathcal{Q}_k} \eta(x)^2.
  \end{equation} 

\end{Lemma}

\begin{proof}
We start with the definition of the Hamiltonian \( H_{\Lambda}(\eta \mid \xi) \) for the pair potential \( \phi \):
\begin{equation*}
\begin{aligned}
    H_{\Lambda}(\eta \mid \xi) &= \frac{1}{2} \sum_{x, y \in \Lambda} \phi(x-y)\eta(x)\eta(y) + \sum_{x \in \Lambda, y \notin \Lambda} \phi(x-y)\eta(x)\xi(y) \\
    &+ \frac{1}{2} \sum_{x, y \notin \Lambda} \phi(x-y)\xi(x)\xi(y).
\end{aligned}
\end{equation*}
Using the finite range (FR) and repulsion (RC) conditions on \( \phi \), we can bound the first term above as follows
\[
\frac{1}{2} \sum_{x, y \in \Lambda} \phi(x-y)\eta(x)\eta(y) \geq A \sum_{j \in \mathcal{K}_{\Lambda}} \sum_{x, y \in \mathcal{Q}_j} \eta(x)\eta(y).
\]
Since \( \phi \) is a positive pair potential, the cross terms involving \( \xi \) are non-negative and thus do not affect the inequality. Therefore, we have
\[
H_{\Lambda}(\eta \mid \xi) \geq A \sum_{j \in \mathcal{K}_{\Lambda}} \eta_{\Lambda}(\mathcal{Q}_j)^2.
\]
For the second inequality, we consider the local Hamiltonian \( H_{\mathcal{Q}_k}(\eta \mid \xi) \)
\[
H_{\mathcal{Q}_k}(\eta \mid \xi) = \frac{1}{2} \sum_{x, y \in \mathcal{Q}_k} \phi(x-y)\eta(x)\eta(y) + \sum_{x \in \mathcal{Q}_k, y \notin \mathcal{Q}_k} \phi(x-y)\eta(x)\xi(y).
\]

Again, using the positivity of \( \phi \), we can ignore the cross terms and obtain:
\[
H_{\mathcal{Q}_k}(\eta \mid \xi) \geq A \sum_{x \in \mathcal{Q}_k} \eta(x)^2.
\]

When \( \xi=0 \), the cross terms vanish, and we are left with
\[
H_{\mathcal{Q}_k}(\eta_k \mid 0) \geq A \sum_{x \in \mathcal{Q}_k} \eta(x)^2.
\]

This completes the proof.
\end{proof}

      \label{lower bound hamiltonian}

\begin{Definition}
  For each $\xi \in \K(\X)$ and $\Lambda \in B_c(\X)$ the partition function is defined
  $$Z_\Lambda(\xi):=\int_{\K(\Lambda)}\exp\{-H_{\Lambda}(\eta_{\Lambda}| \xi)\} \mu_\la^\Lambda (d\eta_{\Lambda}),$$
  where $\mu_\la^\Lambda$ is a  measure, having full support on $\K(\Lambda)$.
\end{Definition}

\begin{Lemma}
\label{partition function lemma}
    Let Assumption \ref{pair potential} hold. Then for any $ \xi \in \K(\X)$ and $\Lambda \in B_c(\X)$ 
    $$ 0<Z_\Lambda(\xi) < \infty. $$
For  $\phi \geq 0$,  $Z_\Lambda(\xi) \leq 1. $
\end{Lemma}
\begin{proof}
Since \( \phi \) is non-negative and symmetric, the Hamiltonian \( H_\Lambda(\eta | \xi) \) is well-defined and finite for all \( \eta \) and \( \xi \). Therefore, \( e^{-H_\Lambda(\eta | \xi)} \) is positive and less than or equal to 1, which implies that \( Z_\Lambda(\xi) > 0 \). The boundedness of \( \phi \) and the finite range condition ensure that \( H_\Lambda(\eta | \xi) \) is bounded above, making \( Z_\Lambda(\xi) \) finite.\\
If \( \phi \geq 0 \), then \( H_\Lambda(\eta | \xi) \geq 0 \) and \( e^{-H_\Lambda(\eta | \xi)} \leq 1 \). Thus, the integral \( Z_\Lambda(\xi) \) is bounded above by the measure of \( \mathcal{K}(\Lambda) \), which is 1. Hence, \( Z_\Lambda(\xi) \leq 1 \).

For each $ \eta, \xi \in \K(\X)$ and $\Lambda \in B_c(\X)$ we have by definition
$$H_{\Lambda}(\eta| \xi):= \int_{\Lambda} \int_{\Lambda} \phi(x,y)\eta(dx)\eta(dy)+2\int_{\Lambda^c} \int_{\Lambda} \phi(x,y)\eta(dx)\xi(dy).$$
By Young's inequality with 
$\varepsilon$  (valid for every 
$\varepsilon > 0$ ) $ab \leq (\dfrac {a^2}{2\varepsilon}+\dfrac{b^2\varepsilon}{2}), a,b \geq 0$ and by definition of family of neighbour cubes \eqref{neighbour} 
\begin{equation*}
\begin{split}
&\int_{\Lambda} \int_{\Lambda} \eta(dx)\eta(dy)+2\int_{\Lambda^c} \int_{\Lambda} \eta(dx)\xi(dy)\\
    &=\sum_{j \in \mathcal{K}_\Lambda}\eta_\Lambda(\mathcal{Q}_j)^2+\sum_{j \in \mathcal{K}_{\Lambda}}\sum_{{l \in \mathcal{K}_{\Lambda}}\cap \partial^\phi j}\eta_\Lambda(\mathcal{Q}_j)\eta_\Lambda(\mathcal{Q}_l)
    +2\sum_{j \in \mathcal{K}_{\Lambda}}\sum_{l \in \mathcal{K}_{\Lambda^c} \cap \partial^\phi j } \eta_\Lambda(\mathcal{Q}_j)\xi_{\Lambda^c}(\mathcal{Q}_l)\\
    &\leq 
    \sum_{j \in \mathcal{K}_\Lambda}\eta_\Lambda(\mathcal{Q}_j)^2 +\sum_{j \in \mathcal{K}_{\Lambda}}m^\phi \eta_\Lambda(\mathcal{Q}_j)^2
   +\bigg(\sum_{j \in \mathcal{K}_{\Lambda}}\dfrac{m^\phi}{\varepsilon}  \eta_\Lambda(\mathcal{Q}_j)^2+\sum_{l \in \mathcal{K}_{\mathcal{U}_\Lambda}}\varepsilon m^\phi  \xi_{\Lambda^c}(\mathcal{Q}_l)^2\bigg)\\
    &=(1+\dfrac{\varepsilon+1}{\varepsilon}m^\phi )\sum_{j \in \mathcal{K}_\Lambda}\eta_\Lambda(\mathcal{Q}_j)^2 +m^\phi\varepsilon\sum_{l \in \mathcal{K}_{\mathcal{U}_\Lambda}}  \xi_{\Lambda^c}(\mathcal{Q}_l)^2.
\end{split}
\end{equation*}
 By definition of partition function and Jensen's inequality 
\begin{equation}
\label{lower bound partition}
    \begin{split}
       &Z_\Lambda(\xi) = \int_{\K(\Lambda)}\exp\{-H_{\Lambda}(\eta| \xi)\} \mu_\la^\Lambda (d\eta) \geq \exp\bigg\{-\int_{\K(\Lambda)}H_{\Lambda}(\eta| \xi)\ \mu_\la^\Lambda (d\eta) \bigg\}\\
       &\geq \exp\bigg\{-||\phi||_\infty \int_{\K(\Lambda)}\bigg[\int_{\Lambda} \int_{\Lambda} \eta(dx)\eta(dy)+2\int_{\Lambda^c} \int_{\Lambda} \eta(dx)\xi(dy)\bigg] \mu_\la^\Lambda (d\eta) \bigg\}\\
       &\geq \exp\bigg\{-||\phi||_\infty \int_{\K(\Lambda)}\bigg[(1+\dfrac{\varepsilon+1}{\varepsilon}m^\phi )\sum_{j \in \mathcal{K}_\Lambda}\eta_\Lambda(\mathcal{Q}_j)^2 +m^\phi\varepsilon\sum_{l \in \mathcal{K}_{\mathcal{U}_\Lambda}}  \xi_{\Lambda^c}(\mathcal{Q}_l)^2\bigg] \mu_\la^\Lambda (d\eta) \bigg\}.
    \end{split}
\end{equation}
\end{proof}
\vspace{10pt}

\subsection{Local Gibbs specification}
\vspace{10pt}
\begin{Definition}
    For each $\Lambda \in B_c(\X )$, the local Gibbs measures with boundary conditions $\xi \in \K(\X)$ are given by 
    $$\ga_{\Lambda}(d\eta|\xi):=\frac{1}{Z_\Lambda(\xi)}e^{-H_{\Lambda}(\eta| \xi)}\mu_\la^\Lambda (d\eta).$$
\end{Definition}
\begin{Definition}
    The local specification $\Pi=\{\pi_\Lambda\}_{\Lambda \in B_c(\X )}$ on $\K(\X)$ is a family of stochastic kernels
    \begin{equation} 
    \label{family of kernels}
        \mathcal{B}(\K(\X))\times \K(\X) \ni (B,\xi) \mapsto \pi_\Lambda(B | \xi) \in [0,1],
    \end{equation}
    where
    $$\pi_\Lambda(B | \xi):=\ga_{\Lambda}(B_{\Lambda, \xi}|\xi)$$
    $$B_{\Lambda, \xi}:=\{\eta_\Lambda \in \K(\Lambda)| \eta_\Lambda+\xi_{\Lambda^c } \in B\}\in \mathcal{B}(\K(\Lambda)).$$
\end{Definition}
\begin{Remark}

The family of kernels  \eqref{family of kernels} has the consistency property , which means that for all $\Lambda, \Tilde{\Lambda}\in B_c(\X )$ with $\Tilde{\Lambda} \subseteq \Lambda$
\begin{equation}
    \label{consistency}
    \int_{\K(\X)}\pi_{\Tilde{\Lambda}} (B | \eta) \pi_\Lambda(d\eta | \xi)=\pi_\Lambda(B | \xi)
\end{equation}
for all $B \in \mathcal{B}(\K(\X))$ and $\xi \in \K(\X) $.
\end{Remark}

\begin{Definition}
    A probability measure $\ga$ on $\K(\X)$ is called Gibbs measure (or state) with pair potential $\phi$ if it satisfies the Dobrushin-Lanford-Ruelle(DLR) equilibrium equation
    \begin{equation}
        \int_{\K(\X)}\pi_\Lambda(B| \eta)\ga(d\eta)=\ga(B)
    \end{equation}
 for all $B \in \mathcal{B}(\K(\X))$ and $\Lambda \in B_c(\X) $.   The associated set of all Gibbs states will be denoted by $G(\phi)$.
\end{Definition}
We are interested in the subset $G^t(\phi)$ of tempered Gibbs measures supported by $\K^t(\X)$. We start defining the set of tempered discrete Radon measures by 
\begin{equation}
    \K^t(\X):=\bigcap_{\alpha>0}\K_\alpha(\X),
\end{equation}
where 
\begin{equation}
    K_\alpha(\X):=\{\eta \in \K(\X)|M_\alpha(\eta)<\infty\} \in \mathcal{B}(\K(\X)),
\end{equation}
with
$$M_\alpha(\eta):=\bigg(\sum_{k \in \mathbb{Z}^d}\eta(Q_k)^2e^{-\alpha|k|}\bigg)^{1/2}.$$
The subset $G^t(\phi)$ of tempered Gibbs measures is defined as
\begin{equation}
    G^t(\phi):=G^t(\phi)\cap \mathcal{P}(\K^t(\X)),
\end{equation}
where $$\mathcal{P}(\K^t(\X)):=\{\ga \in \mathcal{P}(\K(\X)) | \ga(\K^t(\X))=1\}.$$

\section{Existence of Gibbs measures}
The existence of a Gibbs measure \( \gamma \) can be established by showing that the family of finite-dimensional distributions defined by the local specifications is consistent and satisfies the Kolmogorov extension theorem. The conditions of the pair potential \( \phi \) ensure that the interaction is stable and has finite range, which are key properties to apply the Dobrushin-Lanford-Ruelle (DLR) equations.\par
To show that \( G^t(\phi) \) is non-empty, we construct a measure \( \gamma \) by taking the weak limit of a sequence of finite-volume Gibbs measures with suitable boundary conditions. The existence of such a weak limit follows from the Prokhorov's theorem,  given the tightness of the family of finite-volume Gibbs measures, which in turn is a consequence of the properties of \( \phi \).
\begin{Theorem}{(Main Result)}
\label{Main Result}
Let $\phi:\X\times \X \to \R$ be such that Assumption \ref{pair potential} holds. Then there exists a Gibbs measure $\ga$ (at least one) corresponding to the potential $\phi$ and the measure $\mu_\la$, which is supported by $\K^t(\X)$. Therefore,
$$G^t(\phi)\neq \varnothing.$$
Furthermore, the set $G^t(\phi)$ is relatively compact in the topology $\mathcal{T}_\mathcal{Q}$.
\end{Theorem}

\begin{Definition}
    $\mathcal{T}_\mathcal{Q}$ is the topology of $\mathcal{Q}$-local convergence on the space of all probability measures $\mathcal{P}(\K(\X))$.
\end{Definition}
\vspace{10pt}
\subsection{Estimates on the local specification kernels}
\vspace{10pt}
 In order to prove the existence result, first of all we have to prove that local Gibbs specification kernels $\{\pi_\Lambda(d\eta| \xi) \ | \  \Lambda \in \mathcal{Q}_c(\mathbb{R}^d)\}$, with fixed tempered boundary condition is $\xi \in \K^t(\X)$, are locally equicontinuous.
    For that reason, we prove 
    \begin{enumerate}
       \item  exponentially integrability of the following map $ \mathbb k(\X) \ni \eta \mapsto \beta \eta (\mathcal{Q}_k)^2$, for each $k \in  \mathbb Z$, which has a role of Lyapunov functional,
    \item weak dependence of the specification kernels on boundary conditions, both for small volumes $\Lambda:= \mathcal{Q}_k$ and in the thermodynamic limit $ \Lambda \nearrow \X$.
    \end{enumerate}

\begin{Lemma}
\label{exponential}
    For fixed $k \in  \mathbb Z, \ \xi \in \K(\X), \  \Lambda \in \mathcal{B}_c(\X)  $ and $\beta \in [0, A],$

\begin{equation}
    \begin{aligned}
     &\int_{\K(\X)}\exp\bigg\{\beta \eta (\mathcal{Q}_k)^2\}\pi_\Lambda(d\eta| \xi )\\
     &\leq \exp\{\int_{\Lambda} \log (\Psi_\la (\psi(x))) m(dx)\bigg\}
     \times \exp \bigg\{m^\phi \varepsilon||\phi||_\infty\sum_{l \in \mathcal{K}_{\mathcal{U}_\Lambda}}  \xi_{\Lambda^c}(\mathcal{Q}_l)^2 \bigg\},    
    \end{aligned}
\end{equation}
where 
$$
	\Psi_\la (\psi(x)) = \exp(\int_{\mathbb R_0^d} (e^{||\phi||_\infty\big(1+\dfrac{\varepsilon+1}{\varepsilon}m^\phi \big) \sum_{j \in \mathcal{K}_\Lambda}(v \times v)_j\psi_j(x)} -1) \la(dv)).
	$$
\end{Lemma}

\begin{proof}
    By definition of local specification kernel and by Lemma \ref{lower bound} lower bound of local Hamiltonian  we have 
\begin{equation}
\label{lyapunov in}
    \begin{aligned}
        &\int_{\K(\X)}\exp \{\beta \eta (\mathcal{Q}_k)^2\}\pi_\Lambda(d\eta| \xi )\\ 
        &\leq \frac{1}{Z_\Lambda(\xi)}\int_{\K(\X)}\exp\bigg\{-[A-\beta] \eta _\Lambda(\mathcal{Q}_k)^2-A\sum_{j \in \mathcal{K}_\Lambda, j \neq k }\eta _\Lambda(\mathcal{Q}_k)^2 \bigg \}\mu_\la^\Lambda (d\eta) \\
        &\leq \dfrac{\int_{\K(\X)}\exp\bigg\{-[A-\beta] \eta _\Lambda(\mathcal{Q}_k)^2-A\sum_{j \in \mathcal{K}_\Lambda, j \neq k }\eta _\Lambda(\mathcal{Q}_k)^2 \bigg \}\mu_\la^\Lambda (d\eta)  }{\exp\bigg\{-||\phi||_\infty \int_{\K(\Lambda)}(1+\dfrac{1+\varepsilon}{\varepsilon} )\sum_{j \in \mathcal{K}_\Lambda}\eta_\Lambda(\mathcal{Q}_j)^2 +m^\phi \varepsilon\sum_{l \in \mathcal{K}_{\mathcal{U}_\Lambda}}  \xi_{\Lambda^c}(\mathcal{Q}_l)^2\bigg] \mu_\la^\Lambda (d\eta) \bigg\}}\\
        &\leq \exp \bigg\{||\phi||_\infty \int_{\K(\Lambda)}\big[(1+\dfrac{\varepsilon+1}{\varepsilon}m^\phi )\sum_{j \in \mathcal{K}_\Lambda}\eta_\Lambda(\mathcal{Q}_j)^2\big]\mu_\la^\Lambda(d\eta) \\
        &+m^\phi \varepsilon||\phi||_\infty\sum_{l \in \mathcal{K}_{\mathcal{U}_\Lambda}}  \xi_{\Lambda^c}(\mathcal{Q}_l)^2 \bigg\}.
    \end{aligned}
\end{equation}
Using the definition of the Poisson measure, we can define measure $\mu_\Lambda$ by its Laplace transform, similar to \eqref{laplace1} and  \eqref{local 2} for the first term we will obtain
\begin{equation*}
\begin{aligned}
&\int_{\K(\Lambda)} e^{||\phi||_\infty\big(1+\dfrac{\varepsilon+1}{\varepsilon}m^\phi \big)\sum_{j \in \mathcal{K}_\Lambda}<\psi, \eta \times \eta>_j} \mu_\la^\Lambda (d\eta)\\
&= \exp\bigg( \int_{\mathbb R_0^d \times \Lambda} (e^{||\phi||_\infty\big(1+\dfrac{\varepsilon+1}{\varepsilon}m^\phi \big) \sum_{j \in \mathcal{K}_\Lambda}(v \times v)_j\psi_j(x)} -1) \la(dv) m(dx)\bigg).
\end{aligned}
\end{equation*}
Denote 
	$$
	\Psi_\la (r) := \exp(\int_{\mathbb R_0^d} (e^{||\phi||_\infty\big(1+\dfrac{\varepsilon+1}{\varepsilon}m^\phi \big) \sum_{j \in \mathcal{K}_\Lambda}(v \times v)_jr_j} -1) \la(dv)), \;\;r\in\R,
	$$
 or equivalently 
 
\begin{equation}
\label{laplace2}
	\int_{\K(\Lambda)} e^{||\phi||_\infty\big(1+\dfrac{\varepsilon+1}{\varepsilon}m^\phi \big)\sum_{j \in \mathcal{K}_\Lambda}<\psi, \eta \times \eta>_j} \mu_\la^\Lambda (d\eta)= e^{\int_{\Lambda} \log (\Psi_\la (\psi(x))) m(dx)}.
\end{equation}
Plugging \eqref{laplace2} into \eqref{lyapunov in} we obtain
\begin{equation*}
    \begin{aligned}
     \int_{\K(\X)}\exp\bigg\{\beta \eta (\mathcal{Q}_k)^2\}\pi_\Lambda(d\eta| \xi )& \leq \exp\{\int_{\Lambda} \log (\Psi_\la (\psi(x))) m(dx)\bigg\}\\
     &\times \exp \bigg\{m^\phi \varepsilon||\phi||_\infty\sum_{l \in \mathcal{K}_{\mathcal{U}_\Lambda}}  \xi_{\Lambda^c}(\mathcal{Q}_l)^2 \bigg\}.    
    \end{aligned}
\end{equation*}
\end{proof}
\begin{Lemma}
    For fixed $k \in  \mathbb Z, \ \xi \in \K(\X), \  \Lambda \in \mathcal{B}_c(\X)  $ and $\beta \in [0, A],$
\begin{equation}
\label{estimate}
\begin{aligned}
&\int_{\K(\X)}\exp\{\beta \eta (\mathcal{Q}_k)^2\}\pi_k(d\eta| \xi )\\
&\leq \exp\{\int_{\mathcal{Q}_k} \log (\Psi_\la (\psi(x))) m(dx)\bigg\}
     \times \exp \bigg\{ \varepsilon||\phi||_\infty\sum_{j \in \partial^\phi k}  \xi(\mathcal{Q}_j)^2 \bigg\},   
    \end{aligned}
\end{equation}
where 
$$
	\Psi_\la (\psi(x)) = \exp({\int_{\mathbb R_0^d} (e^{||\phi||_\infty(1+\frac{m^\phi}{\varepsilon}) (v \times v)\psi(x)} -1) \la(dv)}).
	$$

\end{Lemma}
\begin{proof}
    Proof follows from Lemma \ref{lower bound} and exploiting the proof of Lemma \ref{exponential}.
\end{proof}

Using estimate \eqref{estimate} we aim to obtain similar moment estimates for arbitrary large domains $\Lambda_\mathcal{K}= \bigsqcup_{k \in \mathcal{K} }\mathcal{Q}_k \in \mathcal{Q}_c(\X)$ indexed by $\mathcal{K} \Subset \mathbb Z ^d$. When $\mathcal{K} \nearrow \mathbb Z ^d$ , $\Lambda_\mathcal{K} \nearrow \X$.

\begin{Proposition}
    Let $\beta \in [0, A]$. Then there exists $\mathcal{C}_\beta< \infty$ such that for all $k \in \mathbb Z ^d$ and $\xi \in \K^t(\X)$
    \begin{equation}
        \label{large estimate}
\limsup_{k \nearrow \mathbb Z ^d} \int_{\K(\X)} \exp{\{\beta \eta (\mathcal{Q}_k)^2 \}}\pi_\mathcal{K}(d\eta| \xi )\leq \mathcal{C}_\beta.
    \end{equation}
Moreover, for each $\alpha>0$ one finds $v_\alpha >0$ and  $\mathcal{C}_\beta< \infty$ such that 
\begin{equation}
        \label{large estimate 2}
\limsup_{k \nearrow \mathbb Z ^d} \int_{\K(\X)} \exp{\{v_\alpha M_\alpha(\eta)^2 \}}\pi_\mathcal{K}(d\eta| \xi )\leq \mathcal{C}_\alpha.
    \end{equation}
\end{Proposition}

\begin{proof}
We define 
    $$0 \leq n_k(\mathcal{K}| \xi):=\log\Bigg\{ \int_{\K(\X)}\exp{\{\beta \eta (\mathcal{Q}_k)^2 \}}\pi_\mathcal{K}(d\eta| \xi )\Bigg\}, \  k \in \mathbb Z ^d $$
which is finite by Lemma \ref{exponential}. \par
At first we aim to find  global bound for the whole sequence $ (n_k(\mathcal{K}| \xi))_{k \in \mathbb Z ^d}$, which means that we have the required estimate on each of its components. Integrating both sides of \eqref{estimate} with respect to $\pi_\mathcal{K}(d\eta| \xi )$ and considering the consistency property   \eqref{consistency}, we obtain the following estimate for any $k \in \mathcal{K}$
\begin{align*}
    n_k(\mathcal{K}| \xi)&\leq log \Bigg \{\int_{\K(\X)}\exp\big[\int_{\mathcal{Q}_k} \log (\Psi_\la (\psi(x))) m(dx)\big]\\ \nonumber
     &\times \exp \big[ \varepsilon||\phi||_\infty\sum_{j \in \partial^\phi k}  \xi(\mathcal{Q}_j)^2\big]\pi_\mathcal{K}(d\eta| \xi ) \bigg\}, 
\end{align*}
 from where we get
\begin{align}
    n_k(\mathcal{K}| \xi)&\leq 
    \int_{\mathcal{Q}_k} \log (\Psi_\la (\psi(x))) m(dx)+\varepsilon||\phi||_\infty \sum_{j \in \mathcal{K}^c \cap\partial^\phi k}  \xi(\mathcal{Q}_j)^2 \nonumber\\
    &+log \bigg\{\int_{\mathbb{K}(\X)}\exp \bigg[ \varepsilon||\phi||_\infty \sum_{j \in \mathcal{K} \cap \partial^\phi k}  \xi(\mathcal{Q}_j)^2 \bigg]\pi_\mathcal{K}(d\eta| \xi )\Bigg\} \label{logintegral 2}.
\end{align}
In order to bound the last summand , we will apply H{\"o}lder inequality
$$\mu\Bigg( \bigsqcap_{j=1}^K f_j^{t_j}\Bigg) \leq \bigsqcap_{j=1}^K \mu^{t_j}(f_j),  \  \mu(f_j):=\int f_j d\mu$$
valid for any probability measure $\mu$, and non-negative functions $f_j$ and $t_j \geq 0$ such that $\sum_{j=1}^K t_j \leq 1, \ k \in \mathbb N$. Choose $\delta, \varepsilon>0$ such that 
\begin{equation}
\label{beta inequality}
  0 < B_\varepsilon:= \varepsilon||\phi||_\infty < \delta \beta  < \beta  \leq \beta_0:=A. 
\end{equation}
In our case $f_j:=\exp\{\beta\eta_j(\X)^2\}$ and 
$$t_j:= \frac{\varepsilon||\phi||_\infty}{\beta}, \ for  \  j \in \mathcal{K}\cap \partial ^\phi k. $$
Then by H{\"o}lder inequality we get that last summand in \eqref{logintegral 2} is dominated by 
\begin{equation}
    \label{dom}
    \sum_{j \in \mathcal{K}\cap \partial ^\phi k }\log \Bigg\{ \int_{\K(\X)}\exp(\beta \eta (\mathcal{Q}_j))\pi_\mathcal{K}(d\eta| \xi ) \Bigg\}^{t_j}=\frac{\varepsilon||\phi||_\infty}{\beta}\sum_{j \in \mathcal{K}\cap \partial ^\phi k }n_j(\mathcal{K}|\xi).
\end{equation}
Let  $k_0 \in \mathbb Z ^d $ be a fixed point in $\mathcal{K}\Subset \mathbb Z ^d$. Let $\vartheta:=R/g+\sqrt{d}$ be such that $|j-k_0|\leq \vartheta$ for all $j \in \partial^\phi{k_0}$. We choose $\alpha>0$ small enough, so that $B_\varepsilon e^{\alpha \vartheta} < \beta$. We take the sum over $k \in \mathcal{K}$ of the terms $n_{k_0}(\mathcal{K}|\xi)$ with the weights $\exp\{-\alpha|k|\}.$ By equations \eqref{logintegral 2} and \eqref{dom} we get for each $k \in \mathcal{K}$
\begin{equation}
\label{n sum}
\begin{aligned}
   n_{k_0}(\mathcal{K}|\xi)& \leq \sum_{k \in \mathcal{K}} [n_k(\mathcal{K}|\xi)\exp\{-\alpha|k|\}]\\
   &\leq \Bigg[ 1- \frac{B_\varepsilon}{\beta}e^{\alpha \vartheta}\Bigg]^{-1}\bigg[ \int_{\mathcal{Q}_k} \log (\Psi_\la (\psi(x))) m(dx)+B_\varepsilon  e^{\alpha \vartheta}  ||\xi_{\mathcal{K}^c}||^2_\alpha\bigg],
   \end{aligned}
\end{equation}
where 
$$\Bigg[ 1- \frac{B_\varepsilon}{\beta}e^{\alpha \vartheta}\Bigg]^{-1}=1+\frac{B_\varepsilon e^{\alpha \vartheta}}{\beta-B_\varepsilon e^{\alpha \vartheta}} \leq \frac{1}{1-\delta e^{\alpha \vartheta}}.$$
Plugging this last inequality in \eqref{n sum}, we obtain 
\begin{equation}
\label{n sum 2}
\begin{aligned}
   n_{k_0}(\mathcal{K}|\xi) &\leq \sum_{k \in \mathcal{K}} [n_k(\mathcal{K}|\xi)\exp\{-\alpha|k|\}]\\
   &\leq \Bigg(\frac{1}{1-\delta e^{\alpha \vartheta}}\Bigg)\bigg[ \int_{\mathcal{Q}_k} \log (\Psi_\la (\psi(x))) m(dx)+B_\varepsilon e^{\alpha \vartheta} M_\alpha(\xi_{\mathcal{K}^c})^2].
   \end{aligned}
\end{equation}
Considering the fact that   $M_\alpha(\xi_{\mathcal{K}^c})$ tends to $0$ as $\mathcal{K}\nearrow \mathbb Z^d$, we obtain
\begin{equation}
\label{limsup}
   \limsup_{k \nearrow \mathbb Z ^d}  \sum_{k \in \mathcal{K}} [n_k(\mathcal{K}|\xi)\exp\{-\alpha|k|\}] \leq \Bigg(\frac{1}{1-\delta e^{\alpha \vartheta}}\Bigg)\int_{\mathcal{Q}_k} \log (\Psi_\la (\psi(x))) m(dx) , 
\end{equation}
letting $\alpha \searrow 0$ we get 
$$\limsup_{k \nearrow \mathbb Z ^d}  n_{k_0}(\mathcal{K}|\xi) \leq  \Bigg(\frac{1}{1-\delta }\Bigg)\int_{\mathcal{Q}_k} \log (\Psi_\la (\psi(x))) m(dx)=:log \mathcal{C}_\beta,$$
which completes the proof.
Thus we have that for each $\beta$ satisfying \eqref{beta inequality} (and therefore for all $\beta \leq A$)
$$\limsup_{k \nearrow \mathbb Z ^d} \int_{\K(\X)} \exp{\{\beta \eta (\mathcal{Q}_k)^2 \}}\pi_\mathcal{K}(d\eta| \xi )\leq \mathcal{C}_\beta.$$
By H{\"o}lder inequality, estimate \eqref{large estimate 2} holds with 
$$v_\alpha:=\beta \Bigg[\sum_{k \in \mathbb Z^d}\exp\{-\alpha|k|\}\Bigg]^{-1}.  $$
Indeed, we have
\begin{align*}
    & \int_{\K(\X)} \exp{\Bigg\{v_\alpha \sum_{k \in \mathbb{Z}^d}\eta(Q_k)^2\exp\{-\alpha|k|\} \Bigg\}}\pi_\mathcal{K}(d\eta| \xi )\\
    & \leq \Bigg(\exp\Bigg\{\sum_{k \in \mathcal{K}} n_k(\mathcal{K}|\xi)\exp\{-\alpha|k|\}  \Bigg\} \Bigg)^{\frac{v_\alpha}{\beta_0}\exp\{-\alpha|k|\} } \exp \{M_\alpha(\xi_{\mathcal{K}^c})^2\}.
\end{align*}
We proof  the second estimate  using \eqref{limsup}  
\begin{equation*}
    \begin{aligned}
        \limsup_{k \nearrow \mathbb Z ^d} &\int_{\K(\X)} \exp{\{v_\alpha M_\alpha(\eta)^2 \}}\pi_\mathcal{K}(d\eta| \xi )\\
        &\leq \exp \Bigg\{\Bigg(\frac{1}{1-\delta e^{\alpha \vartheta}}\Bigg)\int_{\mathcal{Q}_k} \log (\Psi_\la (\psi(x))) m(dx)\Bigg \}=: \mathcal{C}_\alpha.
    \end{aligned}
\end{equation*}

\end{proof}

\begin{Corollary}
\label{support property}
    Let Assumption \ref{pair potential} hold. Then for all $\Lambda \in \mathcal{Q}_c(\X) $ and $N \in \mathbb N$ there exists $\mathcal{C}(\Lambda, N) < \infty$ such that
    $$\limsup_{\substack{\Tilde{\Lambda}\nearrow \X \\ \Tilde{\Lambda} \in \mathcal{Q}_c(\X) }} \int_{\K(\X)}\eta(\Lambda)^N\pi_{\Tilde{\Lambda}}(d\eta| \xi ) \leq \mathcal{C}(\Lambda, N)< \infty,$$
    where $\mathcal{C}(\Lambda, N)$ can be chosen uniformly for all $\xi \in \K^t(\X)$.
\end{Corollary}
The corollary above claims uniform bounds for local Gibbs states.    
\vspace{10pt}
\subsection{Local equicontinuity}
\vspace{10pt}

    \begin{Definition}
        On the space of all probability measures $\mathcal{P}(\mathbb K(\X))$ we introduce the topology of $\mathcal{Q}$-local convergence. This topology, which we denote by $\mathcal{T}_{\mathcal{Q}}$, is defined as the coarsest topology making the maps $\mathcal{P}(\mathbb K(\X)) \ni \ga \mapsto \ga(B)$ continuous for all sets $B \in \mathcal{B}_\mathcal{Q}(\mathbb K(\X))$. Here,
        $$ \mathcal{B}_\mathcal{Q}(\mathbb K(\X)):= \bigcup_{\Lambda \in \mathcal{Q}_c(\X) }\mathcal{B}_\Lambda(\mathbb K(\X))$$
denotes the algebra of all local events associated with the partition cubes, where $\mathcal{B}_\Lambda(\mathbb K(\X))$ $:= \mathbb P_ \Lambda ^{-1}\mathcal{B}(\mathbb K(\Lambda))$ and the canonical projections $\mathbb P$ were defined by \eqref{canonical}.
    \end{Definition}
As we mentioned before, Gibbs measure is constructed as a cluster point of the net of local specification kernels.  A subset of measure from $\mathcal{P}(\mathbb K(\X))$  is relatively compact if and only if each of its nets has a cluster point in $\mathcal{P}(\mathbb K(\X))$, moreover every cluster point is obtained as a limit of a certain subnet. Proving equicontinuity of the local specification is sufficient for the existence of cluster points.
\begin{Definition}
    Fix $\xi \in \mathbb K^t(\X))$. The net $\{\pi_\Lambda(d\eta|\xi) \mid \Lambda \in \mathcal{Q}_c(\X)\}$ is called $\mathcal{Q}$-locally equicontinuous if for all $\tilde{\Lambda} \in \mathcal{Q}_c(\X) $ and for each sequence $\{B_N\}_{N \in \mathbb N } \subset \mathcal{B}_{\tilde{\Lambda}}(\mathbb K(\X))$ with 
    $ B_N \downarrow \varnothing$
    \begin{equation}
        \label{equi def}
        \lim_{N \to \infty}\limsup_{\substack{\Lambda \nearrow \X \\ \Lambda \in \mathcal{Q}_c(\X)}} \pi_\Lambda(B_N| \xi)=0.
    \end{equation}
\end{Definition}

\begin{Proposition}
    Let Assumption \ref{pair potential} hold. Then for each fixed $\xi \in \mathbb K^t(\X))$ the net $\{\pi_\Lambda(d\eta|\xi) \mid \Lambda \in \mathcal{Q}_c(\X)\}$ is locally equicontinuous.
\end{Proposition}
\begin{proof}
    To prove we will split the set $B_N$ and then we will use Corollary \ref{support property}, consistency property \eqref{consistency}, lower bounds for Hamiltonian and partition function.\par
   Let  $\{B_N\}_{N \in \mathbb N } \subset \mathcal{B}_{\tilde{\Lambda}}(\mathbb K(\X))$ be sequence of sets  with $ B_N \downarrow \varnothing$
    and  $\tilde{\Lambda} \in \mathcal{Q}_c(\X)$ be fixed.
    We consider Borel subsets $\mathbb K[\mathcal{U}, T]$ in $\mathbb K(\X)$ containing measures $\eta$ whose local masses over $\mathcal{U}$ are bounded by a given $T>0$,
    $$\mathbb K[\mathcal{U}, T]:=\{\eta \in K(\X) \ | \ \eta(\mathcal{U})\leq T \}, \quad T>0,$$
    where 
    $$\mathcal{U}:=\bigsqcup\{\mathcal{Q}_j \ | \ \exists x \in \mathcal{Q}_j, \exists y \in \tilde{\Lambda}:\  |x-y|\leq R \} \in \mathcal{Q}_c(\X).$$
For each $\xi \in \mathbb K(\X)$ and $\Lambda \in \mathcal{Q}_c(\X)$ containing $\tilde{\Lambda}$, by consistency property we have
\begin{equation}
\label{main}
    \begin{aligned}
        \pi_\Lambda(B_N|\xi)&=\pi_\Lambda(B_N\cap  [\mathbb K(\mathcal{U}, T)]^c | \xi)+\int_{\mathbb K(\X)}\pi_{\tilde{\Lambda}}(B_N\cap  \mathbb K(\mathcal{U}, T) | \eta)\pi_\Lambda(d\eta |\xi)\\
        &=\pi_\Lambda(B_N\cap  [\mathbb K(\mathcal{U}, T)]^c | \xi)+\int_{\mathbb K(\X)}\dfrac{1}{Z_{\tilde{\Lambda}}(\eta)}\int_{\mathbb K(\tilde{\Lambda})}\mathbb{1}_{B_N\cap  \mathbb K(\mathcal{U}, T)}(\rho_{\tilde{\Lambda}}\cup \eta_{\tilde{\Lambda}^c})\\
        &\times \exp\{-H_{\tilde{\Lambda}}(\rho_{\tilde{\Lambda}}|\eta)\}\mu_\lambda^{\tilde{\Lambda}}(d\rho_{\tilde{\Lambda}})\pi_\Lambda(d\eta |\xi).
    \end{aligned}
\end{equation}
For the first summand by Chebyshev's inequality and support property \eqref{support property} we obtain 
\begin{equation*}
    \begin{aligned}
        \limsup_{\substack{\Lambda \nearrow \X \\ \Lambda \in \mathcal{Q}_c(\X)}} \pi_\Lambda(\mathbb K[\mathcal{U}, T]^c)&=\limsup_{\substack{\Lambda \nearrow \X\\\Lambda \in \mathcal{Q}_c(\X)}}\pi_\Lambda(\{\eta \in K(\X) \ | \ \eta(\mathcal{U})> T \}|\xi)\\
        &\leq \limsup_{\substack{\Lambda \nearrow \X\\\Lambda \in \mathcal{Q}_c(\X)}}\int_{\mathbb K(\X)}\dfrac{\eta(\mathcal{U})}{T}\pi_\Lambda(d\eta|\xi) \leq \dfrac{1}{T}\mathcal{C}(\mathcal{U},1) \leq \infty,
   \end{aligned}
\end{equation*}
and therefore
\begin{equation}
\label{first}
    \limsup_{\substack{\Lambda \nearrow \X \\ \Lambda \in \mathcal{Q}_c(\X)}} \pi_\Lambda(\mathbb K[\mathcal{U}, T]^c)=0 \quad as \quad T \nearrow \infty.
\end{equation}
For the second summand, we use Laplace transform and lower bound for Hamiltonian and partition function, 
 \begin{equation}
 \label{second sum}
     \begin{aligned}
        &\dfrac{1}{Z_{\tilde{\Lambda}}(\eta)}\int_{\mathbb K(\tilde{\Lambda})}\mathbb{1}_{B_N\cap  \mathbb K(\mathcal{U}, T)}(\rho_{\tilde{\Lambda}}\cup \eta_{\tilde{\Lambda}^c})\exp\{-H_{\tilde{\Lambda}}(\rho_{\tilde{\Lambda}}|\eta)\}\mu_\lambda^{\tilde{\Lambda}}(d\rho_{\tilde{\Lambda}})\\
        &\leq \exp \{||\phi||_\infty \int_{\K(\tilde{\Lambda})}\big[(1+\dfrac{\varepsilon+1}{\varepsilon}m^\phi )\sum_{j \in \mathcal{K}_{\tilde{\Lambda}} }\rho_{\tilde{\Lambda}}(\mathcal{Q}_j)^2\big]\mu_\lambda^{\tilde{\Lambda}}(d\rho_{\tilde{\Lambda}}) \\
        &+m^\phi \varepsilon||\phi||_\infty\sum_{l \in \mathcal{K}_{\mathcal{U}_{\tilde{\Lambda}}}}  \eta_{\tilde{\Lambda}^c}(\mathcal{Q}_l)^2 \}\int_{\mathbb K(\tilde{\Lambda})}\mathbb{1}_{B_N\cap  \mathbb K(\mathcal{U}, T)}(\rho_{\tilde{\Lambda}}\cup \eta_{\tilde{\Lambda}^c})\\
        &\times\exp\{-A\sum_{j \in \mathcal{K}_{\tilde{\Lambda}} }\rho_{\tilde{\Lambda}}(\mathcal{Q}_j)^2\}\mu_\lambda^{\tilde{\Lambda}}(d\rho_{\tilde{\Lambda}})\\
        &\leq \exp\{m^\phi \varepsilon||\phi||_\infty\sum_{l \in \mathcal{K}_{\mathcal{U}_{\tilde{\Lambda}}}}  \eta_{\tilde{\Lambda}^c}(\mathcal{Q}_l)^2 \} \exp\{\int_{\tilde{\Lambda}} \log (\Psi_\la^\varepsilon (\psi(x))) m(dx)\}\mu_\lambda^{\tilde{\Lambda}}({B_N\cap  \mathbb K(\mathcal{U}, T)}).
     \end{aligned}
 \end{equation}
 Right hand-side of inequality \eqref{second sum} tends to zero as $B_N \searrow \varnothing$, which means 
 \begin{equation}
 \label{second}
     \pi_{\tilde{\Lambda}}(B_N\cap  \mathbb K(\mathcal{U}, T) | \xi) \to 0  \quad as \quad B_N \searrow \varnothing.
 \end{equation}
 Plugging \eqref{first} and \eqref{second} in  \eqref{main} we get equicontinuity of the net $\{\pi_\Lambda(d\eta|\xi) \mid \Lambda \in \mathcal{Q}_c(\X)\}$. 
\end{proof}
We have the following Corollary which will be used to prove the main result.
\begin{Corollary}
\label{convergence}
     Let Assumption \ref{pair potential} hold, fix some order generating sequence $\{\Lambda_N\}_{N \in \mathbb N }$  $ \subset \mathcal{Q}_c(\X)$. Then, for each boundary condition $\xi \in \mathbb K^t(\X))$ a subsequence of $\{\pi_{\Lambda_N}(\cdot \mid \xi)\}_{N \in \mathbb N }$ converges $\mathcal{Q}$-locally to a probability measure $\gamma \in \mathcal{P}(\mathbb K(\X)))$, that means for all $B \in \mathcal{B_\mathcal{Q}}(\mathbb K(\X)),$
     \begin{equation}
     \label{limit}
         \pi_{\Lambda_N}(B|\xi) \to \gamma(B) \quad as \quad N \to \infty. 
     \end{equation}
     
\end{Corollary}
\vspace{8pt}
\begin{proof}[Proof of Theorem \ref{Main Result}]

By Corollary \ref{convergence} we have that subsequence of $\{\pi_{\Lambda_N}(\cdot|\xi)\}_{N \in \mathbb N }$ which converges $\mathcal{Q}$-locally to a probability measure $\gamma \in \mathcal{P}(\mathbb K(\X)))$, where $\Lambda_N \nearrow \X$, $\Lambda_N  \in \mathcal{Q}_c(\X)$ is fixed order generating sequence and $\xi \in \mathbb K^t(\X))$ is a boundary condition. By DLR property, we can derive that $\gamma$ is indeed a Gibbs measure:
\begin{equation}
\label{equations}
\begin{aligned}
    \int_{\K(\X)}\pi_\Lambda(B| \eta)\ga(d\eta)&=\lim_{N \to \infty}\int_{\K(\X)}\pi_\Lambda(B| \eta)\pi_{\Lambda_N}(d\eta|\xi)\\
    &=\lim_{N \to \infty}\pi_{\Lambda_N}(B|\xi)=\gamma(B).
\end{aligned}
\end{equation}
We have  used \eqref{limit}  and consistency property in \eqref{equations}. \\
For the relative compactness of \( G^t(\phi) \), we use the fact that the set of Gibbs measures is tight and closed under the weak topology. Since the space of probability measures on a Polish space is itself Polish under the weak topology, Prokhorov's theorem ensures that tightness implies relative compactness.
Thus, we conclude that \( G^t(\phi) \) is non-empty and relatively compact in the topology \( \mathcal{T}_\mathcal{Q} \).
\end{proof}

\section{Conclusion}
We have introduced the problem of Gibbs measures on the cone of vector-valued Radon measures. We considered positive pair potential and constructed Gibbs measure using the DLR equilibrium equation and a measure $\mu_\lambda$. Then, we constrained it to a set of tempered Gibbs measures. We constructed a measure $\gamma$ as a limit of local specification kernels and proved that the limit measure is indeed Gibbs. Gibbs measure has wide applications in statistical mechanics,  physics, and specific areas of biology, particularly in modelling complex biological systems. 

\section{Aknowledgments}
Luca Di Persio would like to thank Yuri Kondratiev, to whom this work is dedicated.
  I was lucky to meet Yuri about 20 years ago at a crowded conference and immediately recognised a wholly singular, brilliant, inspiring light in him. 
 
 The mathematical relationship has been the excuse to consolidate an experience of human growth that marked me, helping me mature as a man before a mathematician.

  This work is the fruit of an intense collaboration with a true friend, an inspirer and a scientific reference point on whose generous shoulders I hope to continue working.
 \\
 
 Viktorya Vardanyan would like to express her gratitude to Yuri Kondratiev for his generous support. This work would not be possible without his considerable guidance and inspiration.

\end{document}